\documentclass[12pt]{amsart}
\usepackage{array,longtable}
\usepackage[utf8]{inputenc}
\usepackage{bbm}
\usepackage{amssymb,latexsym,fancyhdr,color,graphics}
\usepackage{amsmath, amsthm}
\usepackage{setspace}
\usepackage[dvips]{graphicx}
\usepackage[pdftex]{hyperref}
\usepackage[parfill]{parskip}
\newtheorem{lem}{\bf Lemma}

\newtheorem{prop}{\bf Proposition}
\newtheorem{thm}{\bf Theorem}
\newtheorem{corr}{\bf Corollary}

\newtheorem{rem}{\bf Remark}

\definecolor{titlepagegray}{gray}{0.8}

\begin{document}

\title[Exponential functionals]{On distributions of exponential functionals of the processes with independent increments}

\maketitle

\begin{center}
{\large  L. Vostrikova, LAREMA, D\'epartement de
Math\'ematiques, Universit\'e d'Angers, 2, Bd Lavoisier  49045,
\sc Angers Cedex 01}
\end{center} 
\vspace{0.2in}

\begin{abstract}
The aim of this paper is  to study the laws of the exponential functionals of  the processes $X$ with independent increments, namely
$$I_t= \int _0^t\exp(-X_s)ds, \,\, t\geq 0,$$ 
and also
$$I_{\infty}= \int _0^{\infty}\exp(-X_s)ds.$$
 Under suitable conditions we derive the integro-differential equations for the density of $I_t$ and $I_{\infty}$. We give sufficient conditions for the existence of smooth density of the laws of these functionals. In the  particular case of Levy processes  these equations can be simplified and, in a number of cases, solved explicitly. 
\end{abstract}
%%%%%%%%%%%%%%%%%%%%%%%%%%%%%%%%%%%%%%%%%%%%%%%%%%%%%%%%%%%%%%%%%%%%%%%%%%%%%%
\begin{section}{Introduction}\label{s1}%%%%%%%%%%%%%%%%%%%%%%%%%%%%%%%%%%%%%%%%
%%%%%%%%%%%%%%%%%%%%%%%%%%%%%%%%%%%%%%%%%%%%%%%%%%%%%%%%%%%%%%%%%%%%%%%%%%%%%%
\par This study was inspired by the questions arising in mathematical finance, namely by the questions  related with the perpetuities containing the liabilities, perpetuities subjected to the influence of the economical factors (see, for example, \cite{KR}), and also  with the price of Asian options and similar questions (see, for instance, \cite{CJY},  \cite{V}, and references there). The study of exponential functionals are also important in insurance, since the insurace companies invest the money on risky assets. Then the distributions of these functionals appear very naturally in ruin problem (see for example \cite{A}, \cite{P}, \cite{KP1} and references there). 
\par In mathematical finance the non-homogeneous PII models are more realistic, since price processes are not usually homogeneous processes. For this reason, for example, several authors used for the modelling of log price  a process $X$ such that
$$X_t= \int_0^t g_{s-} dL_s$$
where $L$ is Levy process and $g$ is independent from $L$ c\'{a}dl\`{a}g random process for which the integral is well defined. In this case, the conditional distribution of the process $X$ given $\sigma$~-algebra generated by $g$,
is a PII. Another important example is a Levy process $L$ time changed by an independent increasing process $(\tau_t)_{t\geq 0}$ (cf.\cite{CW}), i.e. $$X_t= L_{\tau_t}.$$ Again, conditionally to the process $\tau$, the process $X$ is PII.
\par In \cite{SV} we obtained the recurrent formulas for Mellin transform and we use these formulas to calculate the moments of exponential functionals of the processes with independent increments. In this paper we obtain the equations for the densities, when they exist, of the laws  of exponential functionals $I_t$ and $I_{\infty}$ of PII processes.
\par Exponential functionals for Levy processes was studied in a big number of articles, the main part of them was related to the study of $I_{\infty}$. The asymptotic behaviour of exponential functionals $I_{\infty}$ was studied in \cite{CPY}, in particular for $\alpha$-stable Levy processes. The authors also give  an integro-differential equation for the density of the law of exponential functionals, when this density w.r.t. Lebesgue measure exists. The questions related with the characterisation of the law of exponential functionals by the moments was studied in \cite{BY}.
\par In more general setting, related to the Lévy case, the following functional
\begin{equation}\label{0} 
 \int _0^{\infty}\exp(-X_{s-})d\eta _s 
\end{equation}
where $X=(X_t)_{t\geq 0}$ and $\eta= (\eta_t)_{t\geq 0}$ are independent Lévy processes, was intensively studied. The conditions for finiteness of the integral \eqref{0} was obtained in \cite{EM}. The continuity properties of the law of this integral was studied in \cite{BLM}, where the authors give the condition for absence of the atoms and also the conditions for absolute continuity of the laws of integral functionals w.r.t. the Lebesgue measure. Under the assumptions  about the existence of the smooth density of these functionals, the equations for the density  are given  in \cite{Be}, \cite{BeL}, \cite{KPS}. 
\par In the papers \cite{PS} and \cite{PRV}, again for Levy process, the  properties of the exponential functionals $I_{\tau_q}$ killed at the independent exponential time $\tau_q$    of the parameter $q>0$, was investigated. In the article \cite{PRV} the authors studied the existence of the density of the law of $I_{\tau_q}$, they give an integral equation for the density and the asymptotics of the law of $I_{\infty}$ at zero and at infinity, when $X$ is a positive subordinator.
The results given in \cite{PS}  involve analytic Wiener-Hopf factorisation, Bernstein functions and contain the conditions for regularity, semi-explicite expression and asymptotics for the distribution function of  $I_{\tau_q}$. Despite numerous studies, the distribution properties of  $I_t$ and $I_{\infty}$ are known only in a limited number of cases. When $X$ is Brownian motion with drift, the distributions of  $I_t$ and $I_{\infty}$ was studied in \cite{D} and  for a big number of  specific processes $X$ and $\eta$, like Brownian motion with drift and compound Poisson process, the distributions of  $I_{\infty}$ was given in \cite{GP}. 
\par Exponential functionals for diffusions stoped at first hitting time was studied in \cite{SW}, where authors  derive  Laplace transform of the functionals  and then, to find their laws,  perform  numerical inversion of Laplace transform. The relations between  hitting times and occupation times for exponential functionals was considered in \cite{SY}, where the versions of identities in law such as Dufresne's identity, Ciesielski-Taylor's identity, Biane's identity, LeGall's identity was given.   
 \par In this article we consider a real valued process $X=(X_t)_{t\geq 0}$ with independent increments  and $X_0=0$, which is a semi-martingale with respect to its natural filtration. We denote by $(B,C,\nu)$ a semi-martingale triplet of this process, which can be chosen deterministic (see \cite{JSh}, Ch. II, p.106). We suppose that
$B=(B_t)_{t\geq 0}$, $C=(C_t)_{t\geq 0}$ and $\nu$ are absolutely continuous with respect to Lebesgue measure, i.e. 
\begin{equation}\label{abscon}
B_t= \int_0^t b_s\,ds, \,\,\,C_t= \int_0^t c_s\,ds, \,\,\, \nu(dt,dx)= dt K_t(dx)
\end{equation}
with c\`{a}dl\`{a}g functions $b = (b_s)_{s\geq 0},c = (c_s)_{s\geq 0},K = (K_s(A))_{s\geq 0, A\in \mathcal{B}(\mathbb{R})}$.
We assume that the compensator of the measure of jumps $\nu$ verify the usual relation: for each $t\in \mathbb{R}^+$
\begin{equation}\label{rt1}
\int_0^t\int_{\mathbb{R}} (x^2\wedge 1)K_s(dx)\,ds  < \infty.
\end{equation}
 For the the main result we will  suppose an additional technical condition that
\begin{equation}\label{rt11}
\int_0^t\int_{|x|>1} e^{|x|}\,K_s(dx)\,ds  < \infty.
\end{equation}
The last condition  implies  that ${\bf E}(|X_t|)<\infty$ for $t>0$ (cf \cite{Sa}, Th. 25.3, p.159) so that  the truncation of the jumps is  no more necessary.
\par We recall that the characteristic function  of $X_t$  $$\phi_t(\lambda)={\bf E}\exp(i\lambda X_t)$$  is defined by the following expression: for $\lambda\in \mathbb{R}$
$$\phi_t(\lambda )= \exp\{i\lambda B_t - \frac{1}{2}\lambda ^2C_t + \int_0^t\int_{\mathbb{R}}(e^{i\lambda x}-1-i\lambda x )\,K_s(dx)\,ds\}.$$
We recall also  that $X$ is a semi-martingale if and only if for all $\lambda\in \mathbb{R}$ the characteristic function  of $X_t$ is of finite variation in $t$ on finite intervals (cf. \cite{JSh}, Ch.2, Th. 4.14, p.106 ). Moreover, the process $X$ always can be written as a sum of a semi-martingale and a deterministic function which is not necessarily of finite variation on finite intervals.
\par The article is organized as follows. The part 2. is devoted to the Kolmogorov type equation  for the law of $I_t$. As known, the exponential functional $(I_t)_{t>0}$  is not a Markov process with respect to the filtration generated by the process $X$. It is continuous increasing process, which prevent the use  of the  stochastic calculus in an efficient way. For these reasons we  fix $t$ and we introduce  a family of stochastic processes $V^{(t)}=(V^{(t)}_s)_{0\leq s\leq t}$ indexed by $t$ and such that $I_t=V^{(t)}_t$ ($P$-a.s.) (see Lemma 1). The construction of such processes is made via the time reversion of the process $X$ at the fixed time $t$,  and this gives us the process $Y^{(t)}=(Y^{(t)}_s)_{0\leq s\leq t}$. We show that if $X$ is PII then $Y^{(t)}$ is also PII, and if $X$ is Levy process, then $Y^{(t)}$ is so (see Lemma 2). We prove then that $V^{(t)}$ is a Markov process
with respect to the natural filtration of the process $Y^{(t)}$ (see Lemma 3), we find the generator and we give the Kolmogorov-type equations for $V^{(t)}$(see Theorem 1). Supposing the existence of the smooth density of the law of the process $V^{(t)}$ we derive the integro-differential equation for the density of the laws of $V^{(t)}_s, 0<s<t$. The density of the law of $V^{(t)}_t$ can be obtained just by integration of the right-hand side of the equation for the density of the law of $V^{(t)}_s$ in $s$ on the interval $]0,t[$ since the laws of $V^{(t)}_t$ and $V^{(t)}_{t-}$ coincide.
The last fact follows from the absence of the predictable jumps of the process $X$, and, hence, the ones of $Y^{(t)}$.
\par In the part 3. we consider the question of existence of the smooth density of the process $V^{(t)}$. The question of the existence of the density of the law of $V^{(t)}_s$, $0\leq s \leq t,$ of the class
$C^{1,2}(]0,t[, \mathbb{R}^{+,*})$ is rather difficult question, which was open question in all cited papers
on exponential functionals. We remark that in the case of Levy processes the laws of $V^{(t)}_t$ and $I_t$ coincide. We will give in Proposition 2  sufficient conditions for the existence of the density of the class
$C^{\infty}(]0,t[,\mathbb{R}^{+,*})$ when $X$ is Levy process. For non-homogeneous diffusions we give a partial answer on this question in Corollary~1.

\par The part 4 is devoted to Levy processes. When $X$ is Levy process, the equations for the density of $I_t$ can be simplified due to the homogeneity (see Proposition 1). We present also the equations for the distribution functions of the laws of $I_t$ and $I_{\infty}$, since these versions have the explicit boundary conditions, which gives the advantages from numerical point of view (cf. Corollary 2). In Corollary 3 we consider very known Brownian case. In  Corollary 4 we  give the equations for the case of Levy processes with integrable jumps, and in Corollary 5, we consider the case of exponential jumps.
\end{section}
%%%%%%%%%%%%%%%%%%%%%%%%%%%%%%%%%%%%%%%%%%%%%%%%%%%%%%%%%%%%%%%%%%%%%%%%%%%%%%%%%%%%%%%
%%%%%%%%%%%%%%%%%%%%%%%%%%%%%%%%%%%%%%%%%%%%%%%%%%%%%%%%%%%%%%%%%%%%%%%%%%%%%%%%%%%%%%%
\begin{section}{Kolmogorov type equation for the density of the law of $I_t$}\label{s1}
%%%%%%%%%%%%%%%%%%%%%%%%%%%%%%%%%%%%%%%%%%%%%%%%%%%%%%%%%%%%%%%%%%%%%%%%%%%%%%%%%%%%%%%
We introduce, for fixed $t>0$,  a time reversal process $Y=(Y_s)_{0\leq s\leq t}$ with $$Y_s= X_t-X_{(t-s)-}.$$ Of course, this process depends on the parameter $t$, but we will omit it for simplicity of the notations.
\par For convenience of the readers we present here  Lemma 1 and Lemma 2 proved in \cite{SV}.
First result establish the relation between $I_t$ and the process $Y=(Y_s)_{0\leq s\leq t}$.
\begin{lem}\label{l1}
For $t>0$ (${\bf P}$-a.s.) 
$$ I_t= e^{-Y_t}\int _0^t e^{Y_s}ds. $$ 
\end{lem}
\it Proof \rm \, We write, using the definition of the process $Y$ and the assumption that $X_0=0$:  
$$e^{-Y_t}\int _0^t e^{Y_s}ds =  \int _0^t e^{-Y_t +Y_s}ds=
\int _0^t e^{-X_{(t-s)_-}+X_0}ds= \int_0^t e^{-X_s}ds = I_t.$$
The last equality holds after time change noticing that the integration of left-hand version and right-hand version of some c\`{a}dl\`{a}g process w.r.t. Lebesgue measure gives the same result.
$\Box$

In the following lemma we claim  that $Y$ is PII and we  precise its semi-martingale triplet. For that we introduce the functions
$\bar{b}= (\bar{b}_u)_{0\leq u \leq t}$, $\bar{c}=(\bar{c}_u)_{0\leq u \leq t}$ and $\bar{K}=(\bar{K}_u)_{0\leq u \leq t}$ putting 
\begin{equation}\label{car1}
\bar{b}_u={\bf 1}_{\{t\}}(u)(b_t-b_0)+ b_{t-u},
\end{equation}
\begin{equation}\label{car2}
\bar{c}_u= {\bf 1}_{\{t\}}(u)(c_t-c_0)+c_{t-u},
\end{equation}
\begin{equation}\label{car3}
\bar{K}_u(x)= {\bf 1}_{\{t\}}(u)(K_t(x)-K_0(x))+ K_{t-u}(x)
\end{equation}
where ${\bf 1}_{\{t\}}$ is indicator function. It means, for instance for $\bar{b}=(\bar{b}_u)_{0\leq u \leq t}$ that
$$\bar{b}_u= \left\{ \begin{array}{lll}
b_{t-u}&\mbox{if}&0\leq u<t,\\
b_t&\mbox{if}&u=t.
\end{array}\right.$$
So,  the function $\bar{b}$ can have a discontinuity at $t$, since in general $b_0\neq~b_t$.
\begin{lem}\label{l2}(cf. \cite{SV})
The process $Y$ is a process with independent increments, it is a semi-martingale with respect to its natural filtration, and its semi-martingale triplet  $(\bar{B}, \bar{C}, \bar{\nu})$ is given by: for $0\leq s \leq t$,
\begin{equation}\label{cat}
\bar{B_s}= \int_0^s \bar{b}_u du, \,\,\bar{C_s}= \int_0^s \bar{c}_u du,\,\,
\bar{\nu}(du,dx) = \bar{K}_u(dx)\,du
\end{equation}

\end{lem}
To obtain an integro-differential equation for the density, we introduce two important processes  related with the process $Y$, namely the process $V=(V_s)_{0\leq s\leq t}$ and $J=(J_s)_{0\leq s\leq t}$ with
\begin{equation}\label{vi}
 V_s= e^{-Y_s}J_s,\hspace{1cm} J_s= \int_0^s e^{Y_u}du.
\end{equation}
We underline that the both  processes depend of the parameter $t$, since it is so for the process $Y$. 
\par We notice that according to Lemma 1, $I_t=V_t$, and then they have the same laws. As we will see, the process $V=(V_s)_{0\leq s\leq t}$ is a Markov process with respect to the natural filtration $\mathbb{F}^Y= (\mathcal{F}_s^Y)_{0\leq s \leq t}$ of the process $Y$ and this fact will help us very much to find the equation for the density of the law of $I_t$.
\begin{lem}The process $V=(V_s)_{0\leq s\leq t}$ is a Markov process with respect to the natural filtration $\mathbb{F}^Y= (\mathcal{F}_s^Y)_{0\leq s \leq t}$ of the process $Y$.
\end{lem}
\begin{proof} We write that for $h>0$
$$V_{s+h}= e^{-Y_{s+h}}\int_0^{s+h} e^{Y_u}du = e^{-(Y_{s+h}-Y_{s})}\,[V_s+\int_s^{s+h}e^{Y_u-Y_s} du]$$
Then for all measurable bounded functions $f$
$${\bf E}(f(V_{s+h})\,|\,\mathcal{F}^Y_s)= {\bf E}\left(f(e^{-(Y_{s+h}-Y_{s})}[V_s+\int_s^{s+h}e^{Y_u-Y_s} du])\,|\,\mathcal{F}^Y_s\right)=$$
$${\bf E}\left(f(e^{-(Y_{s+h}-Y_{s})}[x+\int_s^{s+h}e^{Y_u-Y_s} du])\right)_{|x=V_s}$$
since $Y$ is a process with independent increments. Hence, ${\bf E}(f(V_{s+h})\,|\,\mathcal{F}^Y_s)$ is a measurable function of $V_s$ and we conclude that $V$ is Markov process with respect to the filtration generated by $Y$.
\end{proof}
\par We define the set of functions
$$\mathcal{C}= \{ f\in \mathcal{C}^2_b\,|\,\sup_{y\in\mathbb{R}^+} |f'(y)y|<\infty , \,\,\sup_{y\in\mathbb{R}^+}| f''(y)y^2| < \infty\}$$
and such that $\,f(0)=f'(0)=0$.
For $0\leq s\leq t$ we put
\begin{equation}
\overline{a}_s = -\overline{b}_s + \frac{1}{2}\overline{c}_s+\int_{\mathbb{R}}(e^{-x}-1+x)\overline{K}_s(dx)
\end{equation}
We notice that the conditions \eqref{c2} and \eqref{c3} imply that ( $\lambda$-a.s.)
$$\int_{\mathbb{R}}\,|e^{-x}-1+x|\,\overline{K}_s(dx) < \infty,$$
so that, $\overline{a}_s$ is ( $\lambda$-a.s.) well-defined.
We introduce also for $f\in \mathcal{C}$ the generator  $(\mathcal{A}_s^V)_{0\leq s < t}$ of the process $V$ via
\begin{equation}\label{g}
\mathcal{A}_s^V(f)(y)=
\end{equation}
$$ (1+y\,\overline{a}_s)\,f'(y) + \frac{1}{2}\overline{c}_s\,f''(y) \,y^2 +\int _{\mathbb{R}} \left[f(ye^{-x}) - f(y) - f'(y) y(e^{-x}-1)\right]\overline{K}_s(dx) $$ 
\thm \label{t1} Let us suppose that the conditions \eqref{abscon},\eqref{rt1} and \eqref{rt11} are verified. Then the infinitesimal generator $(\mathcal{A}_s^V)_{0\leq s < t}$ of the Markov process  $V$  is defined by \eqref{g}.
In addition, for $0\leq s\leq t$ and $f\in\mathcal{C}$
\begin{equation}\label{eq11}
{\bf E}(f(V_s)) =  \int_0^s {\bf E}(\, \mathcal{A}^V_u(f)(V_u)\,) du
\end{equation}
where $\mathcal{A}_t^V = \lim_{s\rightarrow t-}\mathcal{A}_s^V$.
If for $0<s\leq t$ the density $p_s$ w.r.t. Lebesgue measure $\lambda$ of the law of  $V_s$ exists and belongs to the class $\mathcal C^{1,2}(]0,t[\times \mathbb{R}^{+,*})$, then $\lambda$-a.s.
\begin{equation} \label{density}
\frac{\partial}{\partial s}p_s(y) = \frac{1}{2} \bar{c}_s \frac{\partial ^2}{\partial y^2}(y^2\,p_s(y)) - \frac{\partial}{\partial y}((\bar{a}_s\,y+1)\, p_s(y)) +
\end{equation}
$$\int _{\mathbb{R}}\left[e^{x}p_s(ye^{x}) - p_s(y) + (e^{-x}-1) \frac{\partial}{\partial y}(y p_s(y))\right]\bar{K}_s(dx) 
$$
 and  the density $p_t$ of the law of $I_t$ verify:
\begin{equation} \label{density1}
p_t(y) = \int_0^t \left\{\frac{1}{2} \bar{c}_s \frac{\partial ^2}{\partial y^2}(y^2\,p_s(y)) - \frac{\partial}{\partial y}((\bar{a}_s\,y+1)\, p_s(y)) +\right.
\end{equation}
$$\left.\int _{\mathbb{R}}\left[e^{x}p_s(ye^{x}) - p_s(y) + (e^{-x}-1) \frac{\partial}{\partial y}(y p_s(y))\right]\bar{K}_s(dx)\right\}ds$$
\rem \rm The existence and the uniqueness of the solution of the integro-differential equation given in Theorem 1 follows from the possibility to identify the characteristics of the corresponding process  from the equation. Since the law of PII process, which is a semi-martingale, is uniquely defined by its characteristics, the solution of  such equation exists and is unique.\\

{\it Proof of Theorem \ref{t1}.\,}\rm  The proof of Theorem 1 will be divided in three parts : in the first part we prove a decomposition \eqref{eq2}, then using limit passage we prove \eqref{eq11},  and finally, in the third part we obtain the equations \eqref{density}, \eqref{density1} for the densities.
\par {\it 1)\,Proof of \eqref{eq2}.} For $f\in\mathcal{C}$ and $0\leq s \leq t$ we write the Ito formula:
\begin{equation}
\label{eq1}
f(V_s)=f(V_0) + \int_0^s f'(V_{u-}) dV_u + \frac{1}{2}\int_0^s f''(V_{u-})d<V^c>_u+
\end{equation}
$$\hspace{3cm}\int_0^s \int_{\mathbb{R}} \left(f(V_{u-}+x)- f(V_{u-}) - f'(V_{u-})x\right)\mu_V(du,dx),$$
where $\mu_V$ is the measure of jumps of the process $V$. From the definition of the process $V$ we can easily find  that
$$dV_s= ds + J_s\, d(\,e^{-Y_s}\,)$$
and that
\begin{equation}\label{eq22}
 dV^c_s = - e^{-Y_{s-}}J_{s}\, dY^c_s=-V_{s-}\,dY^c_s,
\end{equation}
$$d<V^c>_s= (e^{-Y_{s-}}J_{s})^2 \,d<Y^c>_s = V_{s-}^2 \,d<Y^c>_s,$$
$$ \Delta V_s= V_s-V_{s-}= e^ {-Y_{s-}}J_{s}(e^{-\Delta Y_s}-1)= V_{s-}(e^{-\Delta Y_s}-1).$$
At the same time, again by the Ito formula we get the following decomposition
\begin{equation}\label{eq3}
e^{-Y_s} =  e^{-Y_0} +A_s + N_s.
\end{equation}
In this decomposition the process  $(A_s)_{s\geq 0}$ is defined via 
\begin{equation}\label{dec1}
A_s= \int_0^s e^{-Y_{u-}}[-\bar{b}_u+\frac{1}{2}\bar{c}_u+ \int_{\mathbb{R}}(e^{-x}-1+x)\bar{K}_u(dx)]\, du
\end{equation}
and it is a process of locally bounded vatiation on bounded intervals. In fact, let us
introduce a sequence of stopping times: for $n\geq 1$
$$\tau _n= \inf\{ 0\leq s \leq t\,| \, e^{-Y_s}\geq n\}$$
with $\inf\{\emptyset\}=\infty$. We notice that this sequence of stopping times tends to $+\infty$ as $n\rightarrow \infty$. Then, since $e^{Y_{s-}}<n$ on the stochastic interval $[\!\![ 0,\tau_n[\!\![$, 
we get from \eqref{abscon}, \eqref{rt1} and \eqref{rt11} that
$$\mbox{Var}(A)_{s\wedge \tau _n}\leq n\int_0^t \left[ |\bar{b}_u| + \frac{1}{2}\bar{c}_u + \int_{\mathbb{R}}\,|e^{-x}-1+x|\,\overline{K}_u(dx)\right] du < \infty .$$
In \eqref{eq3} the process $N=(N_s)_{s\geq 0}$ is defined by
\begin{equation}\label{dec2}
\hspace{-5cm}N_s=- \int_0^s e^{-Y_{u-}}d\bar{M}_u+
\end{equation}
$$\hspace{+3cm}\int_0^s \int_{\mathbb{R}}e^{-Y_{u-}}(e^{-x}-1+x)(\mu _Y(du,dx)-\bar{K}_u(dx) du)$$
In the relation \eqref{dec2}, the process $\bar{M}$ is the martingale component of the semi-martingale decomposition of $Y$: $ Y_s= \bar{B}_s + \bar{M}_s$, and $\mu_Y$ is the measure of jumps  of the process $Y$. It should be noticed that since $Y$ is a process with independent increments and $\bar{B}$ is deterministic, $\bar{M}$ is a martingale (see \cite{ShCh}, Th. 58, p. 45) as well as its pure discontinuous part $\bar{M}^d$.
Then, the process  $(N_{s\wedge \tau _n})_{0\leq s \leq t}$ is a local martingale as a stochastic integral of a bounded function w.r.t. a martingale.
\par We put \eqref{eq22},\eqref{eq3},\eqref{dec1} and \eqref{dec2} into \eqref{eq1} to obtain a final decomposition for $f(V_s)$.
To present this final decomposition we put  for  $y\geq 0$ and $x\in \mathbb{R}$
$$F(y,x)=f(ye^{-x}) - f(y) - f'(y) y(e^{-x}-1) $$ and also
$$B^V_s= \int_0^s \left[f'(V_{u-})(1+\bar{a}_u\,V_{u-}) + \frac{1}{2} f''(V_{u-})V_{u-}^2\,\bar{c}_u +\int_{\mathbb{R}} F(V_{u-},x) \bar{K}_u(dx)\,\right]du$$
and 
$$N^V_s= \int_0^s f'(V_{u-})\,V_{u-}[-d\bar{M}_u + \int_{\mathbb{R}}(e^{-x}-1+x)(\mu _Y(du,dx)-\bar{K}_u(dx)\,du]+$$
$$\int_0^s\int_{\mathbb{R}}F(V_{u-},x)(\mu_Y(du,dx)- \bar{K}_u(dx)\,du).$$

Finally, we get a decomposition
\begin{equation}\label{eq2}
f(V_s)=f(V_0) + B^V_s + N^V_s
\end{equation}
In this decomposition $B^V$ is a process with locally bounded variation and $N^V$ is a local martingale. In fact, let us use the same sequence of stopping times $\tau _n$ as previously and let
 $$D= \sup_{y\in\mathbb{R}}\max (|f(y)|, |f'(y)|, |f'(y)y|, |f''(y)y^2|).$$ 
Then,
$$\mbox{Var}(B^V)_{s\wedge \tau _n}\leq \,D \int_0^t \left(1+ |\bar{b}_u| + \bar{c}_u + \int_{\mathbb{R}}\,|e^{-x}-1+x|\,\overline{K}_u(x)dx\right) du $$
$$ + \int_0^t\int_{\mathbb{R}} |F(V_{u-},x)| \bar{K}_u(dx)\,du$$
The first term of the r.h.s. is finite since the functions $(\bar{B}_s)_{0\leq s\leq t}$ and $(\bar{C}_s)_{0\leq s\leq t}$  have finite variation on finite intervals and since \eqref{rt11} holds. Now,
using the Taylor-Lagrange formula of the second order, we find that for $y>0$ and $|x|\leq 1$
$$|F(y,x)|= \frac{1}{2}|f''(y(1+\theta(e^{-x}-1)))|y^2(e^{-x}-1)^2\leq \frac{D(e^{-x}-1)^2}{2[1+\theta(e^{-x}-1)]^2}$$
where $0<\theta<1$. Since for $|x|<1$, $1+\theta(e^{-x}-1)\geq \frac{1}{e}$ and $|e^{-x}-1|\leq e |x|$, we find that $|F(V_{u-},x)|\leq \frac{1}{2} D e^4x^2$ . 
\par For $|x|>1$ we use Taylor-Lagrange formula of the first order to get
$$|F(y,x)|= |(f'(y(1+\theta(e^{-x}-1)))- f'(y)) y(e^{-x}-1)|\leq $$
$$D\,\left[|1+\theta(e^{-x}-1)|^{-1}+1\right]\,|e^{-x}-1|$$
Again, for $x>1$,
$1+\theta (e^{-x}-1)\geq e^{-x}$, and for $x<-1$, $1+\theta(e^{-x}-1)\geq 1$. Moreover, for $x>1$, $|e^{-x}-1|\leq 1$ and for $x<-1$, $|e^{-x}-1|\leq e^{-x}$. 
Finally, $$|F(y,x)|\leq C \left( e^{|x|}{\bf 1}_{\{|x|>1\}}+ x^2\,{\bf 1}_{\{|x|\leq 1\}}\right)$$ with some positive constant $C$. Then, the conditions \eqref{rt1} and \eqref{rt11} implies that
$$\int_0^t\int_{\mathbb{R}} |F(V_{u-},x)| \bar{K}_u(dx)\,du < \infty .$$
Using above results we see that $(N^V_{s\wedge\tau_n})_{0\leq s\leq t}$ is a local martingale  as an integral of a bounded function w.r.t. a martingale.
\par {\it 2) Proof of \eqref{eq11}}. Let $(\tau'_n)_{n\in\mathbb{N}}$ be the localising sequence for $N_V$ and $\bar{\tau_n}= \tau_n\wedge\tau'_n$. Let $s\in [0,t[$ and $\delta>0$ such that $s+\delta \leq t$. Then, from previous decomposition using the localisation we get:
$${\bf E}(f(V_{(s+\delta)\wedge\bar{\tau_n}}) - f(V_{s\wedge\bar{\tau_n}})\,|\, \mathcal{F}^Y_s) = {\bf E}(B^V_{(s+\delta)\wedge\bar{\tau_n}} - B^V_{s\wedge\bar{\tau_n}}\,|\, \mathcal{F}^Y_s)  $$
Since $f$ is bounded function and $\lim_{n\rightarrow \infty} \bar{\tau_n}= + \infty$, we can pass to the limit in the l.h.s. by the Lebesgue convergence theorem. The same can be done on the r.h.s. since the process $B^V= (B^V_s)_{0\leq s\leq t}$ is a process of bounded variation on bounded intervals, uniformly in s and n.
After taking a limit as $n\rightarrow +\infty$ we get that
$${\bf E}(f(V_{s+\delta}) - f(V_s)\,|\, \mathcal{F}^Y_s) = {\bf E}(B^V_{s+\delta} - B^V_s\,|\, \mathcal{F}^Y_s)  $$
Now, we write the expression for $B^V_{s+\delta} - B^V_s$:
$$\hspace{-4cm} B^V_{s+\delta} - B^V_s= \int_s^{s+\delta} [f'(V_{u-})(1+\bar{a}_u\,V_{u-}) + $$
$$\hspace{4cm}\frac{1}{2} f''(V_{u-})V_{u-}^2\,\bar{c}_u +\int_{\mathbb{R}} F(V_{u-},x) \bar{K}_u(dx)\,]du$$
We remark that
$$\lim_{\delta\rightarrow 0}\frac{B^V_{s+\delta} - B^V_s}{\delta}=f'(V_{s-})(1+\bar{a}_s\,V_{s-}) + \frac{1}{2} f''(V_{s-})V_{s-}^2\,\bar{c}_s +\int_{\mathbb{R}} F(V_{s-},x) \bar{K}_s(dx)$$
We show that the  quantities $\frac{B^V_{s+\delta} - B^V_s}{\delta}$ are uniformly bounded, for small $\delta >0$, by a constant.  In fact, we can write that
$$ \frac{\mid B^V_{s+\delta} - B^V_s \mid}{\delta}\leq \frac{C}{\delta}\int_s^{s+\delta}\left[ \left( 1+\bar{a}_u + \frac{1}{2}\bar{c}_u\right)  + \int_{\mathbb{R}}| F(V_{u-},x)| \bar{K}_u(dx)\right] du $$
Then, using the estimations for $|F(V_{s-},x)|$ obtained previously, and the fact that the sequences
$$\frac{1}{\delta}\int_s^{s+\delta} \bar{a}_u du,\, \frac{1}{\delta}\int_s^{s+\delta} \bar{c}_u du,\,
\frac{1}{\delta}\int_s^{s+\delta}\!\!\int_{\mathbb{R}}\left( x^2 I_{\{|x|\leq 1\}}+ e^{|x|} I_{\{|x|>1\}}\right) \bar{K}_u(dx) du$$ are uniformly bounded, for small  values of $\delta >0$, by a constant, we deduce that 
the quantities $\frac{\mid B^V_{s+\delta} - B^V_s \mid}{\delta}$ are uniformly bounded for small $\delta$ by a constant, too.
Under these conditions we can exchange
the limit and the conditional expectation and it gives us the expression for the generator of $V$ at $0\leq s <t$.
As a conclusion, we get that for $0\leq s <t$
\begin{equation}\label{difeq} 
\frac{d}{ds}{\bf E}(f(V_s))=  {\bf E}\mathcal{A}^V_s(f)(V_{s-})
\end{equation}
Let us prove that we can replace $V_{s-}$ by $V_s$ in the above expression. In fact, for $\lambda\in \mathbb{R}$
$${\bf E}(e^{i\lambda\ln(\frac{V_s}{V_{s-}})}) = {\bf E}(e^{-i\lambda \Delta Y_s})= \lim_{h\rightarrow 0+}
{\bf E}(e^{-i\lambda ( Y_{s+h} - Y_s)})=1$$
since the characteristics of $Y$ are continuous in time. Hence, $V_{s-}=V_s$ (P-a.s.) and they have the same laws.
\par Then after the replacement of $V_{s-}$ by $V_s$  in \eqref{difeq} and the integration w.r.t. $s$  we obtain \eqref{eq11}. For $s=t$ we take $\lim_{s\rightarrow t-}$ in \eqref{eq11}.

\par {\it 3) Proof of \eqref{density} and \eqref{density1}.} We denote by $P_s$ the law of $V_s$. Then from \eqref{eq11} we get that for $0\leq s \leq t$
\begin{equation}\label{fin}
\int_0^s\int _0^{\infty}\left[f'(y)(1+y\,\bar{a}_u) + \frac{1}{2}f''(y) y^2 \,\bar{c}_u+\right.
\end{equation}
$$\left.\int _{\mathbb{R}} \left(f(ye^{-x}) - f(y) - f'(y) y(e^{-x}-1)\right)\bar{K}_u(x) dx\right]P_u(dy)\,du\,=\, \int _0^{\infty}f(y)P_s(dy)$$
\par Moreover, since for $s>0$, the law $P_s$ of $V_s$ has a density $p_s$ w.r.t. Lebesque measure, it gives
\begin{equation}\label{fin1}
\int_0^s\int _0^{\infty}\left[f'(y)(1+y\,\bar{a}_u) + \frac{1}{2}f''(y) y^2 \,\bar{c}_u+\right.
\end{equation}
$$\left.\int _{\mathbb{R}} \left(f(ye^{-x}) - f(y) - f'(y) y(e^{-x}-1)\right)\bar{K}_u(x) dx\right]p_u(y)dy\,du $$
$$=\, \int_0^{\infty}f(y)p_s(y) dy$$
To obtain the equation for the density, we consider  the set of continuously differentiable  functions on compact support $\mathcal{C}^{2}_K\subseteq \mathcal{C}$.
We differentiate the above equation with respect to $s$ to get
\begin{equation}\label{fin2}
\int _0^{\infty}\left[f'(y)(1+y\,\bar{a}_s) + \frac{1}{2}f''(y) y^2 \,\bar{c}_s+\right.
\end{equation}
$$\left.\int _{\mathbb{R}} \left(f(ye^{-x}) - f(y) - f'(y) y(e^{-x}-1)\right)\bar{K}_s(x) dx\right]p_s(y)\,$$
$$=\, \int_0^{\infty}f(y)\,\frac{\partial}{\partial s}p_s(y)\, dy$$
Using the integration by part formula we  deduce that
$$\int _0^{\infty} f'(y)\,p_s(y)dy = -\int _0^{\infty}\frac{\partial}{\partial y}(p_s(y)) \,f(y) dy$$
$$\int _0^{\infty} f'(y)\,y\,p_s(y)dy =  -\int _0^{\infty}\frac{\partial}{\partial y}(y\,p_s(y))\, f(y) dy$$
$$\int _0^{\infty} f''(y)\, y^2 \,p_s(y)dy = \int _0^{\infty}\frac{\partial^2}{\partial y^2}(y^2\,p_s(y)) \,f(y) dy $$
By the change of the variables and by the integration by parts we obtain
$$\int _0^{\infty}\int _{\mathbb{R}} p_s(y)\,\left[f(ye^{-x}) - f(y) - f'(y) y(e^{-x}-1)\right]\bar{K}_s(dx)dy=$$
$$\int _0^{\infty}\left(\int _{\mathbb{R}}[e^{x}p_s(ye^{x}) - p_s(y)+ (e^{-x}-1)\frac{\partial}{\partial y} (yp_s(y)] \,\bar{K}_s(dx) \right)\,f(y)dy$$
The mentioned relations together with the equation \eqref{fin2} gives that for all $f\in \mathcal{C}^2_K$ :
$$\int_0^{\infty} f(y) \left[-\frac{\partial}{\partial s} p_s(y) + \frac{1}{2} \bar{c}_s \frac{\partial}{\partial y}( y^2p_s(y)) - \frac{\partial}{\partial y} (( \bar{a}_sy+1)p_s(y)\right. $$
$$\left.+ \int _{\mathbb{R}} e^xp_s(ye^x)-p_s(y)+(e^{-x}-1)\frac{\partial}{\partial y}(yp_s(y))\bar{K}_s(dx) \right]=0$$
and it proves our claim about the equation for $p_s$.
\par We integrate the equation for $p_s$ on the interval $]0, t-\delta[$ for $\delta>0$ and we pass to the limit as $\delta\rightarrow 0$. Since the laws of $V_{t-}$ and $V_t$ coincide, we get in this way the equation for $p_t$.
$\Box$
\section{Some results about the existence of the smooth density}
 The question of the existence of the density of the law of $V_s$, $0\leq s \leq t,$ of the class
$C^{1,2}(]0,t[, \mathbb{R})$ is rather difficult question, which was open question in all cited works
on exponential functionals. We will give here a partial answer on this question via the known result on Malliavin calculus given in \cite{BGJ}. For the convenience of the readers we present this result here in the one-dimensional case.
\par We consider the following stochastic differential equation:
$$ X^x_t= x + \int_0^t a(X_{s-}^x)\,ds + \int_0^t b(X_{s-}^x)dW_s + \int_0^t\int_{\mathbb{R}} c(X_{s-}^x, z)(\mu-\nu)(ds,dz)$$
where $x\in \mathbb{R}$, $a,b,c$ are measurable functions on $\mathbb{R}$ and $\mathbb{R}^2$ respectively, $W$ is standard Brownian motion, and $\mu$ and $\nu$ are jump measure and its compensator of $X^x$. It is assumed that the solution of this equation exists and is unique, and also that the following assumptions hold.
\par {\it Assumption} (A-r):
\begin{enumerate}
\item[(i)] $a$ and $b$ are $r$-times differentiable with bounded derivatives of all order from 1 to $r$,
\item[(ii)]$c(\cdot, z)$ is $r$-times differentiable and there exists a sigma-finite measure $G$ on $\mathbb{R}$ such that
\begin{enumerate} 
\item $c(0,\cdot)\in \bigcap_{2\leq p<\infty}L^p(\mathbb{R^*}, G)$
\item  for $1\leq n\leq r$, $\sup_{y}(\frac{\partial^n}{\partial y^n}(c(y,\cdot ))\in \bigcap_{2\leq p<\infty}L^p(\mathbb{R^*}, G)$
\end{enumerate}
\end{enumerate}
\par {\it Assumption} (SB-$(\zeta,\theta$)): there exist $\epsilon>0$ and $\delta>0$ such that 
$$b^2(y)\geq \frac{\epsilon}{1+|y|^{\delta}}$$
\par {\it Assumption} (SC-bis) : for all $u\in[0,1]$ there exists $\zeta>0$ such that
$$|1+u\frac{\partial}{\partial y}c(y,z)| >\zeta$$
\par {\bf Theorem 2.29} (cf. \cite{BGJ}, p. 15) Suppose that the assumptions (A-(2r+10)), (SB-$(\zeta,\theta$)) and (SC-bis) are satisfied. Then for $t>0$ the law of  $X^x_t$ has a density $p_t(x,y)$ w.r.t. Lebesgue measure and the map $(t,x,y)\rightarrow p_t(x,y)$ is of class ${\it C}^r(]0,t]\times \mathbb{R}\times \mathbb{R})$.
\par To apply this theorem let us write stochastic differential equation for $(V_s)_{0\leq s\leq t}$. For that we put for $0\leq s\leq t$
$$\left\{\begin{array}{l}
a_s(y) = y(-\bar{b}_s+\frac{1}{2}\bar{c}_s + \int_{\mathbb{R}}(e^{-z}-1+z)\bar{K}_s(dz) +1,\\
b_s(y)= y\sqrt{\bar{c}_s},\\
c_s(y,z)= y(e^{-z}-1).
\end{array}\right.$$
\prop Suppose that 
$$\int_0^t\int_{\mathbb{R}}|e^{-z}-1+z|\bar{K}_s(dz)<\infty$$
and that  $\bar{c_s} >0$ for $0< s\leq t$.
Then the process $(V_s)_{0\leq s\leq t}$ satisfy the following stochastic differential equation:
\begin{equation}\label{sde}
V_s= \int_0^s a_u(V_{u-})du - \int_0^s b_u(V_{u-})d W_u+ \int_0^s \int_{\mathbb{R}}c_u(V_{u-},z)(d\mu _Y- \bar{K}_u(dz)du )
\end{equation}
where $\mu _Y$ is the jump measure of $Y$ and $W$ is DDS Brownian motion corresponding to the continuous martingale part $Y^c$ of $Y$.
\begin{proof} We recall that $V_s$ is defined by \eqref{vi}. Let us introduce the process $\hat{Y}$ via the relation : for $0\leq s\leq t$
\begin{equation}\label{dd}
e^{-Y_s}= \mathcal{E}(\hat{Y})_s
\end{equation}
where $\mathcal{E}(\cdot)$ is Dol\'{e}an-Dade exponential.
Then,
$$V_s= \mathcal{E}(\hat{Y})_s\int_0^s \frac{du}{\mathcal{E}(\hat{Y})_u}$$
and we can see by the integration by part formula that $(V_s)_{0\leq s\leq t}$ is unique strong solution of the equation
\begin{equation}\label{eqvi}
dV_s= V_{s-}d\hat{Y}_s + ds
\end{equation} 
with the initial condition $V_0=0$.
Using the definition of Dol\'{e}an-Dade  exponential we see that \eqref{dd} is equivalent to
$$e^{-Y_s}= e^{\hat{Y}_s-\frac{1}{2}<\hat{Y}^c>}\prod_{0<u\leq s}(1+\Delta\hat{Y}_u)\,e^{-\Delta\hat{Y}_u} $$
where $\hat{Y}^c$ is continuous martingale part of $\hat{Y}$.
From this equality we find that $\hat{Y}^c_s= -Y^c_s$, $\ln(1+\Delta\hat{Y}_s)= - \Delta Y_s$
and that the semi-martingale characteristics $(\hat{B},\hat{C},\hat{\nu})$ of $\hat{Y}$ are:
$$\left\{\begin{array}{l}
\hat{B}_s= -\bar{B}_s+\frac{1}{2}\bar{C}_s + \int_0^s \int_{\mathbb{R}}(e^{-z}-1+z)\bar{K}_u(dz)du\\
\hat{C}_s= \bar{C}_s\\
\hat{\nu}(ds,dz) = (e^{-z}-1)\bar{K}_s(dz) ds
\end{array}\right.$$
Since $(\bar{B},\bar{C},\bar{\nu})$ are absolutely continuous w.r.t. Lebesgue measure with the derivatives $(\bar{b},\bar{c},\bar{K})$ we get that
$$\hat{Y}_s= \int_0^s( -\bar{b}_u+\frac{1}{2}\bar{c}_u +\int_{\mathbb{R}}(e^{-z}-1+z)\bar{K}_u(dz))du- \int_0^s \sqrt{\bar{c}_u} \,dW_u+$$
$$\hspace{3cm} \int_0^s\int_{\mathbb{R}}(e^{-z}-1)(\mu_Y(du,dz)- \bar{K}_u(dz)du)$$
where $W$ is DDS Brownian motion corresponding to the continuous martingale part of $Y$.
Let us put this decomposition into \eqref{eqvi} and we obtain \eqref{sde}. 
\end{proof}\rm
\par  To apply the Theorem 2.29 from \cite{BGJ} we suppose  that $Y$ is a Levy process and we introduce the supplementary process 
$$ V_s^{x} = x+ \mathcal{E}(\hat{Y})_s\int_0^s \frac{du}{\mathcal{E}(\hat{Y})_u}$$
with $x>0$.
We see that $V_s^x-x=V_s$, and $V^0_s=V_s$, and also that the density $p_s(x,y)$ of the law of $V_s^x$ w.r.t. Lebesgue measure and the density $p_s(y)$ of the law of $V_s$ exist or not at the same time and are related: for all $x>0$ and $y>0$
$$p_s(x, y+x)= p_s(y).$$
So, the both densities are the same regularity w.r.t.$(s,y)$.
\prop Suppose that $Y$ is Levy process with the triplet $(b_0,c_0,K_0)$ and the following conditions are satisfied:
\begin{enumerate}
\item $c_0>0$,
\item  $\int_{z<-1}e^{-pz}\,K_0(dz)<\infty$ for $p\geq 2$,
\item there exists a constant $A>0$ such that $K_0(]A,+\infty[)=0$.
\end{enumerate}
Then, for $s>0$, the law of $V_s$ has a density $p_s$  and the map $(s,y)\rightarrow p_s(y)$ is of class $C^{\infty}(]0,t],\mathbb{R}^{+,*})$.
\begin{proof}When $Y$ is Levy process, the functions $a_s, b_s, c_s$ figured in \eqref{sde} are independent on $s$ and are equal to:
$$\left\{\begin{array}{l}
a(y) = y(-b_0+\frac{1}{2}c_0 - \int_{\mathbb{R}}(e^{-z}-1+z)K_0(dz))+1,\\
b(y)= y\sqrt{c_0},\\
c(y,z)= y(e^{-z}-1).
\end{array}\right.$$
We consider the process $V^x$ with $x>0$. We see that the Assumption (A-r) is satisfied for all $r\geq 1$ with $G=K_0$, as well as the Assumption (SB-$(\zeta,\theta))$ putting 
$\epsilon= x^2c_0$ and (SC-bis) taking $\zeta= \frac{1}{2}e^{-A}$. Then the map $(s,x,y)\rightarrow p_s(x,y)$
is of class $C^{\infty}(]0,t],\mathbb{R}^{+,*}\times\mathbb{R}^{+,*})$, and the map $(s,y)\rightarrow p_s(y)$ is of class $C^{\infty}(]0,t],\mathbb{R}^{+,*})$
\end{proof}\rm

\par For non-homogeneous diffusion we obtain the following partial result.
\corr Let $s>0$ be fixed. Suppose that
\begin{enumerate}
\item $\int_0^s\bar{c}_u du >0$,
\item  $\int_0^s\int_{z<-1}e^{-pz}\,K_s(dz)<\infty$ for $p\geq 2$,
\item there exist a constant $A>0$ such that  $K_s(]A,+\infty[)=0$ for   all $0<s\leq t$.
\end{enumerate}
Then, the law of $V_s$  has a density $p_s$  such that the map $y\rightarrow p_s(y)$ is of class $C^{\infty}(\mathbb{R}^{+,*})$.
\begin{proof} We notice that the law of $Y_s$ coincide with the law of Levy process with the triplet
$(\frac{1}{s}\bar{B}_s, \frac{1}{s}\bar{C}_s, \frac{1}{s}\int_0^s\int_{\mathbb{R}}\bar{K}_s(dz) du)$
at the time $s$. Then the previous proposition can be applied and it gives the claim.
\end{proof}
\end{section}

\begin{section}{When $X$ is Levy process}
In this section we consider a particular case of Levy processes. Namely, let $X$ be Levy process with the parameters $(b_0,c_0,K_0)$. 
As before we suppose that
 \begin{equation}\label{condlevy}
\int_{\mathbb{R}}(x^2\wedge 1)K_0(dx) < \infty \,\,\mbox{and}\,\, \int_{|x|>1}\,e^{|x|}K_0(dx) < \infty 
 \end{equation}
 and we put
$$a_0 = -b_0 + \frac{1}{2}c_0+\int_{\mathbb{R}}(e^{-x}-1+x)K_0(dx) .$$
\par  Due to the homogeneity of Levy process, the equation for the density can be simplified as we can see from the following proposition.
\prop\label{p1}
Suppose that \eqref{condlevy} holds and the density of the law of $I_t$ exists and belongs to the class $\mathcal C^{1,2}(]0,t]\times \mathbb{R}^{+,*})$. Then this density verify the following equation:
\begin{equation} \label{density2}
\frac{\partial}{\partial t}p_t(y) = \frac{1}{2} c_0 \frac{\partial ^2}{\partial y^2}(y^2\,p_t(y)) - \frac{\partial}{\partial y}((a_0y+1)\, p_t(y)) +
\end{equation}
$$\hspace{3cm}\int _{\mathbb{R}}\left[e^{x}p_t(ye^{x}) - p_t(y) + (e^{-x}-1) \frac{\partial}{\partial y}(y p_t(y))\right]K_0(dx)
$$
In particular case, when $I_{\infty}<\infty$ ($P$-a.s.) and the density $p_{\infty}$ of the law of $I_{\infty}$ exists and belongs to the class $\mathcal{C}^2(\mathbb{R}^{+,*})$, we have
\begin{equation} \label{density3}
\hspace{-2cm}\frac{1}{2} c_0 \frac{d^2}{dy^2}(y^2\,p_{\infty}(y)) - \frac{d}{dy}((a_0y+1)\, p_{\infty}(y)) +
\end{equation}
$$\hspace{2cm}\int _{\mathbb{R}}\left[e^{x}p_{\infty}(ye^{x}) - p_{\infty}(y) + (e^{-x}-1) \frac{d}{dy}(y p_{\infty}(y))\right]K_0(dx) = 0
$$

\begin{proof} In the case of Lévy processes we write that ($P$-a.s.)
$$V_s= e^{-Y_s}\,J_s = e^{-X_t+X_{(t-s)-}}\,\int_0^s e^{X_t-X_{(t-u)-}}du =$$
$$\int_0^s e^{X_{(t-s)-}-X_{(t-u)-}}du = \int_0^s e^{X_{(t-s)}-X_{(t-u)}}du.$$
Due to the homogeneity of the Lévy processes we have the following identity in law:
$$\mathcal{L}((X_{t-u}-X_{t-s})_{0\leq s\leq t}) = \mathcal{L}((X_{s-u})_{0\leq s\leq t})$$
Then,
$$\mathcal{L}(\int_0^s e^{X_{(t-s)}-X_{(t-u)}}du) = \mathcal{L}(\int_0^s e^{-X_{(s-u)}}du)= \mathcal{L}(\int_0^s e^{-X_{u}}du)$$
where the last equality is obtained by time change. As a conclusion, $\mathcal{L}(V_s)= \mathcal{L}(I_s)$ for $0\leq s\leq t$, and, hence, $(p_s)_{0< s\leq t}$ are  the densities of the laws of 
$(I_s)_{0<s\leq t}$. In addition, again due to the homogeneity, for all $0\leq s \leq t$, $\bar{b}_s= b_{t-s}=b_0$, $\bar{c}_s= c_{t-s}=c_0$, $\bar{K}_s(dx)= K_{t-s}(dx)=K_0(dx)$. Then, from the Theorem 1 we obtain \eqref{density2}. 
\par Again due to the homogeneity, for $0<s\leq t$, the generator $\mathcal{A}^V_s(f) = \mathcal{A}(f)$ where
$$ \hspace{-3cm}\mathcal{A}(f)(y) = (1+y\,a_0)\,f'(y) + $$
$$\frac{1}{2}c_0\,f''(y) \,y^2 +\int _{\mathbb{R}} \left[f(ye^{-x}) - f(y) - f'(y) y(e^{-x}-1)\right]K_0(dx)$$ 
and it does not depend on $s$. Moreover, $\mathcal{L}((V_s)_{0\leq s \leq t})= \mathcal{L}((I_s)_{0\leq s \leq t})$ and the equality \eqref{eq11} become
$${\bf E}f(I_s) =  \int_0^s {\bf E}\mathcal{A}(f)(I_u)du.$$
We divide the both sides of the above equality by $s$ and we let $s$ go to infinity. Since $f$ is bounded, we get zero as a limit on the l.h.s.. Since $I_t\stackrel{P-p.s.}{\longrightarrow}I_{_\infty}$, we also get for $f\in \mathcal{C}$
$$\lim_{s\rightarrow\infty}{\bf E}\mathcal{A}(f)(I_{s})= {\bf E}\mathcal{A}(f)(I_{\infty}).$$
Then,  $  {\bf E}\mathcal{A}(f)(I_{\infty})=0$  and we obtain  \eqref{density3} in the same way as in Theorem 1, by the integration by parts and the time change.
\end{proof}

\rem \rm Similar equation for the density of $I_{\infty}$ in the case when $\int_{\mathbb{R}}(|x|\wedge 1)\,K_0(dx) < \infty$ was obtained in \cite{CPY}. We recall that the condition on $K_0$ of \cite{CPY}  is stronger at zero then our condition.  It should be mentioned that the authors of \cite{CPY}  did not obtained the equation for the density of $I_t$.

\corr \label{c1} Under the conditions of Proposition \ref{p1}, the distribution function $F_t$ of $I_{t}$ verify second order integro-differential equation 
\begin{equation}
\label{eqt}\frac{\partial}{\partial t}F_t(y) = \frac{1}{2} c_0 \frac{\partial}{\partial y}(y^2\, \frac{\partial}{\partial y}F_t(y)) - (a_0y+1)\,  \frac{\partial}{\partial y}F_t(y)+
\end{equation}
$$\hspace{3cm}\int _{\mathbb{R}}\left[F_t(ye^{x}) - F_t(y) + (e^{-x}-1)\,y\,  \frac{\partial}{\partial y}F_t(y))\right]K_0(dx) $$
with following boundary conditions:
$$F_t(0)=0,\,\,\,\lim _{y\rightarrow +\infty}F_t(y)=1.$$
When $I_{\infty}<\infty$, the similar equation is valid for the distribution function $F_{\infty}$
of the law of $I_{\infty}$:
\begin{equation}\label{eqinfty}
\frac{1}{2} c_0 \frac{d}{dy}(y^2\,F'_{\infty}(y)) - (a_0y+1)\, F'_{\infty}(y)+
\end{equation}
$$\int _{\mathbb{R}}\left[F_{\infty}(ye^{x}) - F_{\infty}(y) + (e^{-x}-1)\,y\, F'_{\infty}(y))\right]K_0(dx) =0$$
with similar boundary conditions:
$$F_{\infty}(0)=0,\,\,\,\lim _{y\rightarrow +\infty}F_{\infty}(y)=1.$$
\begin{proof} We integrate each term of the equation of Theorem \ref{t1} on $[0,y]$ and we use the fact that
$$\int_0^y p_{t}(u)du = F_{t}(y)$$
since $F_{t}(0)=0$. We take in account the fact that that the map $(t,u)\rightarrow p_t(u)$ is of class $ C^{1.2}(\mathbb{R}^{+,*}\times \mathbb{R}^{+,*})$ to exchange the integration and the derivation. The same we do for $F_{\infty}(y)$.
\end{proof}
\corr \label{c2} (cf. \cite{D}, \cite{BS}) Let us consider  Brownian motion with drift, i.e.
$$ dX_t= b_0dt +\sqrt{c_0} dW_t$$
where $c_0\neq 0$ and $b_0>0$.
Then the law of exponential functional $I_t$ associated with $X$ has a density which verify :
$$\frac{\partial}{\partial t}p_t(y) = \frac{1}{2} c_0 \frac{\partial ^2}{\partial y^2}(y^2\,p_t(y))- \frac{\partial}{\partial y}((a_0y+1)\, p_t(y))$$
In particular, for $I_{\infty}$ we get that
$$p_{\infty}(x) = \frac{1}{\Gamma (\frac{2\,b_0}{c_0})\,x}\,\left(\frac{2}{c_0 x}\right)^{\frac{2b_0}{c_0}}\,\exp\left(-\frac{2}{c_0x}\right)$$
\begin{proof} From the Proposition \ref{p1} we find the equation for $p_t$. From  Corollary  1  we get the equation for $F_{\infty}$ :
$$\frac{1}{2} c_0 \frac{d}{dy}(y^2\,F'_{\infty}(y)) - (a_0y+1)\, F'_{\infty}(y)=0$$
This equation is equivalent to
$$\frac{1}{2} c_0 y^2\,F''_{\infty}(y) - ((a_0 -c_0)y+1)\, F'_{\infty}(y)=0$$
By the reduction of the order of the equation, we find that
$$F'_{\infty}(y)= C \,y^{2(\frac{a_0}{c_0}-1)}\,\exp\left(-\frac{2}{c_0y}\right)$$
with some positive constant $C$.
Using boundary conditions we calculate a constant $C$. We get that $C= \frac{1}{\Gamma\left(1-\frac{2a_0}{c_0}\right)}\left(\frac{c_0}{2}\right)^{2\frac{a_0}{c_0}-1}$
where $\Gamma (\cdot )$ is gamma function. Since $1-\frac{2a_0}{c_0}=\frac{2b_0}{c_0}$, this gives us the final result. 
\end{proof}

\rem \rm We recall that the distribution of $I_t$ for Brownian motion with drift was obtained in \cite{D} and in \cite{BS}, formula 1.10.4, p. 264.
\par Let us denote by $\nu^+$ and $\nu^-$ the Levy measure of positive and negative jumps respectively, namely for $x>0$
$$\nu^+([x,+\infty[)=\int _x^{+\infty}K_0(du),\,\,\,\nu^-(]-\infty, -x])=\int _{-\infty}^{-x}K_0(du)$$
To simplify the notations we put also $$\nu^+(x)=\nu^+([x,+\infty[),\,\nu^-(x)= \nu^-(]-\infty ,-x])$$
Let us suppose in addition that
$$\int_{\mathbb{R}}|x|\,K_0(dx)<\infty .$$
\corr \label{c3} Suppose that $X$ is a Levy process  with integrable jumps.
Then, under the conditions of Proposition \ref{p1}, the density $p_{t}$ of $I_{t}$, verify~:
$$\frac{\partial}{\partial t}p_t(y) =\frac{1}{2}c_0 \frac{\partial ^2}{\partial y^2}(y^2\,p_{t}(y)) - \frac{\partial}{\partial y}((r_0y+1)\, p_{t}(y))+$$
$$\int _y^{+\infty}p_{t}(z)\nu ^+ (\ln(\frac{z}{y}))\,dz +\int_0^y p_{t}(z) 
\nu ^- (-\ln(\frac{z}{y}))\,dz$$
where $r_0= a_0-\int_{\mathbb{R}}(e^{-x}-1)\,K_0(dx)= -b_0+\frac{1}{2}c_0+\int_{\mathbb{R}}x\,K_0(dx)$.\\
\begin{proof} We take the equation \eqref{eqt} and we rewrite it in the following form:
$$ \frac{\partial}{\partial t}F_t(y) = \frac{1}{2} c_0 \frac{\partial}{\partial y}(y^2\, \frac{\partial}{\partial y}F_t(y)) - (r_0y+1)\,  \frac{\partial}{\partial y}F_t(y)+
$$
$$\hspace{3cm}\int _{\mathbb{R}}\left[F_t(ye^{x}) - F_t(y)\,\right]K_0(dx) $$
Then we divide the integral over $\mathbb{R}$ in two parts  integrating on $]0,+\infty[$ and $]-\infty , 0[$. We do the integration by parts :
$$\int_{\mathbb{R}}[F_{t}(ye^{x}) - F_{t}(y)]K_0(dx) =$$
$$\int_0^{+\infty}\frac{\partial}{\partial x}F_{t}(ye^{x})\,ye^{x}\,\nu^+(x)dx +\int_{-\infty}^0 \frac{\partial}{\partial x}F_{t}\,(ye^{x})\,ye^{x}\,\nu^-(x)dx             $$
 and we change the variables $z= ye^x$.  We differentiate the result w.r.t. $t$ and this gives us the claim.
\end{proof}
\corr \label{c4} Suppose that for $x\in \mathbb{R}$
$$K_0(x)= e^{-\mu x}I_{\{x>0\}}$$
Then, under the assumptions of Proposition \ref{p1}, the density $p_{t}$ of $I_{t}$, verify :
$$\frac{\partial}{\partial t}p_t(y) = \frac{1}{2}c_0 \frac{\partial ^2}{\partial y^2}(y^2 \,p_{t}(y))-\frac{\partial}{\partial y}((r_0y+1)\, p_{t}(y)) + \frac{y^{\mu}}{\mu}\int_y^{\infty}\frac{p_t(z)}{z^{\mu}}\,dz$$
In particular, for the density $p_{\infty}$ of $I_{\infty}$   we have :
$$ \frac{1}{2}c_0 \frac{\partial ^2}{\partial y^2}(y^2 \,p_{t}(y))-\frac{\partial}{\partial y}((r_0y+1)\, p_{t}(y)) + \frac{y^{\mu}}{\mu}\int_y^{\infty}\frac{p_t(z)}{z^{\mu}}\,dz=0$$
\begin{proof}
We take in account that $\nu^+(x)= \frac{1}{\mu}e^{-\mu x}$ and $\nu^-(x)=0$ for all $x>0$ and this gives us the equation for $p_t$ in this particular case.
\end{proof}
\end{section}

\end{document}